\documentclass[11pt, one side, article]{memoir}
\usepackage{coll}

\title{Collectives: Compositional protocols for\\contributions and returns}
\author{Nelson Niu and David I. Spivak}
\date{}
\begin{document}

\maketitle

\chapter{Introduction}
\label{sec:intro}

We introduce a concept called a \emph{collective}: an interface with a protocol for aggregating contributions and distributing returns.
Through such a protocol, many members may participate in a mutual endeavor.%
\footnote{The idea has a very simple mathematical description---a collective is a $\otimes$-monoid in $\poly$, the category of polynomial functors---but we do not require the reader to have any background in polynomial functors to understand this note.
We direct those who would like to learn more to \cite{poly_book}, in preparation.}

Here is one such collective: each member contributes an amount of money, and these are pooled into a single investment; then a return on this investment is divided among the members in a manner proportionate to their contributions. 
For example, if three members contribute \$2, \$3, and \$5, and the return is \$20, it would be distributed back to the members as \$4, \$6, and \$10. 

In general, a collective constitutes how contributions are aggregated and how returns are accordingly distributed. The above simple case generalizes in many ways: the sort of thing members can contribute---time, work, ideas, resources, etc.---as well as the way these contributions are aggregated and returns distributed can be accordingly rich. The world of collectives is meant to be explored through real-world needs.

We will now give the definition of collective, though we caution that there are several equations that may be difficult to parse and intuit at first. Do not fear: the rest of the paper is designed to explain and explore how collectives work in a variety of examples, as well as provide some theory for building new collectives from old. Let's dive in.

\begin{definition}\label{def.collective}
    A \emph{collective} $\mathscr{C}$ consists of
    \begin{itemize}
        \item a set $C$ of \emph{contributions}; and,
        \item for each $c\in C$, a set $R[c]$ of \emph{returns} on $c$;
    \end{itemize}
    with the following operations:
    \begin{itemize}
        \item a \emph{neutral contribution} $\e\in C$;
        \item an \emph{aggregation} operation $\ast\colon C\times C\to C$ which is associative and has unit $\e$;
        \item for each $a,b\in C$, a \emph{distribution} operation \[\left(\frac{a}{a\ast b}, \frac{b}{a\ast b}\right): R[a\ast b]\to R[a]\times R[b]\] 
        that satisfies the following ``cancellation'' equations:%
        \footnote{For any two functions $f\colon X\to Y$ and $g\colon Y\to Z$, we use $f\then g$, read as ``$f$, then $g$,'' to denote their composite $X\to Z$.}
 \begin{gather}
    \frac{a}{a\ast \e}=\id_{R[a]}=\frac{a}{\e\ast a}\label{gath.eq1}\\
    \frac{a\ast b}{(a\ast b)\ast c}\then \frac{a}{a\ast b}=
    \frac{a}{a\ast(b\ast c)}
    \qquad
    \frac{b\ast c}{a\ast(b\ast c)}\then \frac{c}{b\ast c}=
    \frac{c}{(a\ast b)\ast c}\label{gath.eq2}\\
    \frac{a\ast b}{(a\ast b)\ast c}\then \frac{b}{a\ast b}=
    \frac{b\ast c}{a\ast(b\ast c)}\then\frac{b}{b\ast c}\label{gath.eq3}
 \end{gather}
    \end{itemize}
    
We say that a collective is \emph{commutative} if both of the following conditions are met: its aggregation satisfies
\begin{equation}\label{eqn.comm_agg}
    a\ast b=b\ast a \qquad \text{for all } a,b\in C
\end{equation}
(i.e.\ its aggregation is commutative), and its distribution satisfies
\begin{equation} \label{eqn.comm_dis}
    \frac{a}{a\ast b}=\frac{a}{b\ast a} \qquad \text{for all } a,b\in C.
\end{equation}
\end{definition}

We move straight to an example before explaining much about the definition; we recommend ignoring aspects of the notation that are in any way difficult and simply getting an idea for the nature of a collective. Afterward, we will return to parsing the definition.

\begin{example}[Prediction market]
\label{ex.prediction_market}
    Imagine a prediction market as a collective comprised of teams of analysts considering a set $E$ of candidates. Each team contributes a probability distribution $p$ on $E$, which represents how likely they believe each candidate is to win.\footnote{We denote the set of probability distributions on $E$ by $\Dist(E)$.} When some external process chooses a winning candidate $e\in E$ and a total reward to be returned, this winner is communicated to each team, and the reward is paid out according to the probability each team assigned to $e$. 
    
    In this setting, a contribution is a pair $(k,p)$ where $k\in \nn$ is a number of analysts and $p:E\to [0,1]$ is their consensus distribution. For each contribution $(k,p)$ the set of returns is $R[(k,p)]:= E\times \rr_+$, so that a single return is a pair $(e,r)$, where $e\in E$ is a candidate and $r\in\rr_+$ is a positive reward.

    The protocol of the prediction market is then given as an algebraic structure on this interface: any pair of contributions is combined by the weighted average of the predicted probabilities
    \[(k,p)\oplus (l,q) :=\left(k+l, \text{avg}_{k,l}(p,q)\right),\quad\text{where}\quad\text{avg}_{k,l}(p,q):=\frac{kp+lq}{k+l}.\]
    Here we use $\oplus$ to denote the aggregation.
    When a winner $e$ and reward $r$ are chosen, the winner is communicated to each team, and the reward is divided up and distributed in proportion with the contributions:
    \[R[(k,p)\oplus (l,q)]\to R[(k,p)]\times R[(l,q)]\]
    \[(e,r)\mapsto \left(\left(e,\frac{kp(e)r}{kp(e)+lq(e)}\right),\left(e,\frac{lq(e)r}{kp(e)+lq(e)}\right)\right)\]
This operation is ``coassociative,'' by which we mean that if several analysts make up each team, and several teams make up a super-team, then dividing rewards to each team within the super-team and then to each analyst within each team is the same as going all the way from the super-team directly to the analysts who make it up.
Moreover, the contribution satisfies \eqref{eqn.comm_agg} and the distribution satisfies \eqref{eqn.comm_dis}, so this is a commutative collective: its protocol does not depend on the order of the contributions.

In the above case (and it is not unusual in this regard) we can see that although the fraction notation is not literal, it nicely represents the actual distribution function:
\[\frac{(k,p)}{(k,p)\oplus(l,q)}(e,r):=\left(e,\frac{kp(e)r}{kp(e)+lq(e)}\right)\text{ and }\frac{(l,q)}{(k,p)\oplus(l,q)}(e,r):=\left(e,\frac{lq(e)r}{kp(e)+lq(e)}\right).\]
\end{example}

With \cref{ex.prediction_market} guiding intuition, we now move to explain \cref{def.collective} in terms of both notation and concrete interpretation.

\begin{notation}
    Note that $\frac{a}{a\ast b}$ is a purely formal notation for ``the first member's portion of the return, when the first member contributes $a$ and the second member contributes $b$.'' The fraction notation is meant to evoke this:
    \[
    \frac{\cancel{a\ast b}}{(a\ast b)\ast c}\then \frac{a}{\cancel{a\ast b}}=
    \frac{a}{a\ast(b\ast c)}
    \]
     even though it does not always represent a fraction or quotienting operation.
     
     Note that $\frac{a}{a\ast a}$ is ambiguous: it could mean the left or the right distribution. Thus we require that the two variables in the denominator are denoted differently, so the above would be denoted either $\frac{a_1}{a_1\ast a_2}$ or $\frac{a_2}{a_1\ast a_2}$, with $a_1=a_2$ as a side condition.
\end{notation}

\begin{remark}
Cancellation equations \eqref{gath.eq1} come from the diagrams
\[
    \begin{tikzcd}
    R[a*\e]\ar[r, "\frac{a}{a\ast \e}"]\ar[dr, equal]&
    R[a]\times R[\e]\ar[d, "\pi_1"]\\&
    R[a]
    \end{tikzcd}
\qquad
    \begin{tikzcd}
    R[\e]\times R[a]\ar[d, "\pi_2"']&R[\e\ast a]\ar[l, "\frac{a}{\e\ast a}"
    ']\ar[dl, equal]\\
    R[a]
    \end{tikzcd}
\]
and cancellation equations in \eqref{gath.eq2} and \eqref{gath.eq3} come from the diagram
   \[\begin{tikzcd}
	{R[a\ast b\ast c]} && {R[a\ast b]\times R[c]} \\
	\\
	{R[a]\times R[b\ast c]} && {R[a]\times R[b]\times R[c].}
	\arrow["{\left(\frac{a}{a\ast(b\ast c)}, \frac{b\ast c}{a\ast(b\ast c)}\right)}"', from=1-1, to=3-1]
	\arrow["{\left(\frac{a{\ast} b}{(a \ast b)\ast c}, \frac{c}{(a\ast b)\ast c}\right)}", from=1-1, to=1-3]
	\arrow["{\left(\frac{a}{a\ast b}, \frac{b}{a\ast b}\right)\times R[c]}", from=1-3, to=3-3]
	\arrow["{R[a]\times\left(\frac{b}{b \ast c}, \frac{c}{b\ast c}\right)}"', from=3-1, to=3-3]
\end{tikzcd}\]
We will refer to \eqref{gath.eq2} and \eqref{gath.eq3} together as the \emph{coassociativity} of the distribution function.
\end{remark}

\begin{remark}
    Notice that the set $C$ together with the neutral contribution $\e\in C$ and the aggregation $*\colon C\times C\to C$ form a familiar algebraic structure called a \emph{monoid}: a set endowed with an associative binary operation that has a unit.
    Aggregating contributions is associative, but note that it is not a priori commutative unless \eqref{eqn.comm_agg} holds: for many collectives, the order of the contributions matters.
    Nevertheless, associativity suggests that we can think of the aggregation function 
    \[*: C^n\to C\] 
    as an operation that takes $n$ inputs, corresponding to any number $n\in\nn$ of members in the collective.
    
    A collective, then, is a monoid $C$ endowed with extra structure: a set associated with each $c\in C$ and an operation on these sets going in the direction \emph{opposite} to the monoid operation, yet in a coherent way.
    If a monoid gives a way of aggregating contributions associatively, the extra structure of a collective tells us the possible returns on each contribution, as well as how to distribute any return on an aggregate contribution ``coassociatively'' to the members in a manner compatible with the aggregation method.
    
    Like with the aggregation function, the coassociativity of the distribution function allows us to think of it as an operation that gives $n$ outputs; or, equivalently, $n$ functions with the same domain, each of the form
    \[
        \frac{a_i}{a_1\ast\cdots\ast a_i\ast\cdots a_n}:R[a_1\ast\cdots\ast a_n]\to R[a_i]
    \]
    for some $i$ from $1$ to $n$. The coassociativity ensures that there is no ambiguity here. 
    
    Even when the aggregation function is commutative, the distribution function may not be symmetric: there exist collectives satisfying \eqref{eqn.comm_agg} but not \eqref{eqn.comm_dis}, i.e.\ collectives that are not commutative even though their underlying monoids are commutative.
    Such a collective remembers the order in which contributions were made and uses this order in the distribution protocol, even when the aggregation protocol has forgotten it.
    See \cref{ex.comm_list_monoid} for an example.
    \end{remark}

\begin{notation} \label{not.interface}
\delme{
We call the set of contributions together with the set of returns on each contribution the \emph{interface} of the collective.
We denote the interface as a polynomial or a power series in a single variable, $\yon$, as follows. Instead of numbers, our coefficients will be singleton sets of contributions, and our exponents will be sets the corresponding sets of returns:%
\[
    \sum_{c\in C}\{c\}\yon^{R[c]}.
\]
Some useful sets of returns are the ordinals: let $\ord{n}$ denote the set with $n$ elements $\ord{n}=\{0,\ldots,n-1\}$.

When clear from context, we may omit the coefficient sets and simply write
\[
    \sum_{c\in C}\yon^{R[c]}.
\]
Much like with standard polynomials, the coefficient in front of $\yon^{R[c]}$ should be thought of as a $\1$, a singleton set---which, after all, is isomorphic to the omitted coefficient $\{c\}$.

In cases where multiple contributions have the same set of returns, we combine terms with the same exponent by taking the (disjoint) union of their coefficients.
For example, we can write an interface with contributions $C=\{c_1,c_2,c_3,d_1,d_2\}$ with return sets $R$ on each $c_i$ and $S$ on each $d_j$ as
\[
    \{c_1\}\yon^{R}+\{c_2\}\yon^{R}+\{c_3\}\yon^{R}+\{d_1\}\yon^{S}+\{d_2\}\yon^{S}=\{c_1,c_2,c_3\}\yon^R+\{d_1,d_2\}\yon^S.
\]
As another example, the interface of a collective in which every set of returns is $R$ can be written as $C\yon^R$.

This polynomial representation of the interface of a collective is rather like a generating function. It is notation that is convenient for packaging up the data of a collective.
Later, in \cref{sec:making}, we will see how multiplication and composition of polynomials correspond to very natural operations on collectives with those interfaces.
}
We call the set of contributions together with the set of returns on each contribution the \emph{interface} of the collective.
We denote the interface as a polynomial or a power series in a single variable, $\yon$, as follows. Instead of numbers, our exponents will be sets---specifically, sets of returns:%
\footnote{Using a polynomial as we do in \eqref{eq.poly_notation} could be viewed as analogous to \emph{generating functions} in combinatorics. It is a convenient way to package the interface (the contributions and returns) of a collective. Indeed, later in \cref{sec:making} we will see how multiplication and composition of polynomials correspond to very natural operations on collectives with those interfaces. The reader will hopefully get used to the notation as we explore it through examples.}
\begin{equation}\label{eq.poly_notation}
\sum_{c\in C}\yon^{R[c]}.
\end{equation}
Just as we can rewrite a standard polynomial $\yon^2+\yon$ as $1\yon^2+1\yon$, we can rewrite our interface polynomials with singleton sets as coefficients, explicitly tagging each term with its corresponding contribution:
\[
\sum_{c\in C}\{c\}\yon^{R[c]}.
\]
This can aid clarity when writing a polynomial without the $\sum$ sign, e.g.\ $\{\text{alice}\}\yon^{\nn}+\{\text{bob}\}\yon^{\{a,b,c\}}+\{\text{carla}\}\yon^{\{a,b,c\}}$.
Much like with standard polynomials, we can condense this notation by ``combining like terms'' (terms with the same exponents)---we take the disjoint union of their coefficients: $\{\text{alice}\}\yon^{\nn}+\{\text{bob}\}\yon^{\{a,b,c\}}+\{\text{carla}\}\yon^{\{a,b,c\}}\iso\{\text{alice}\}\yon^{\nn}+\{\text{bob},\text{carla}\}\yon^{\{a,b,c\}}$.
As another example of the same phenomenon, the interface of a collective in which every set of returns is $R$ can be written as $C\yon^R\cong\sum_{c\in C}\yon^R$.

For $n\in\nn$ we denote the set $\{0,1,\ldots,n-1\}$ by $\ord{n}$. So $\0=\varnothing$, $\1$ is a singleton set, etc.
\end{notation}

\section*{Acknowledgments}
We thank Christian Williams for numerous insightful discussions and suggestions.
We credit \cref{ex.prediction_market} to Spencer Breiner and \cref{ex.probabilities} to David Jaz Myers, while \cref{ex.lists,ex.potluck} were inspired by a conversation with Owen Lynch.
This material is based upon work supported by the Air Force Office of Scientific Research under award number FA9550-20-1-0348.

\chapter{Collective examples}
\label{sec:examples}

\begin{table}[]
    \centering
    \begin{tabular}{l|l|l|l}
         \textbf{Ref.}&\textbf{Carrier}&\textbf{Conditions}&\textbf{Summary} \\\hline
         \ref{ex.prediction_market}&
            $\nn\Dist(E)\yon^{E\rr_+}$&
            $E\in\smset$&
            Reward good predictors\\
        \ref{ex.monoid_linear}&
            $M\yon$&
            $(M,\1,*)$ monoid&
            Aggregate only\\
        \ref{ex.set_repr}&
            $\yon^S$&
            $S\in\smset$&
            Distribute only\\
        \ref{ex.investment_reals}&
            $\yon+\rr_+\yon^{\rr_+}$&
            ---&
            Proportional rewards\\
        \ref{ex.ylist_monoid}&
            $\sum_{x\in\rr_{\geq0}}\yon^{[0,x]}$&
            ---&
            Queue (continuous, homog.)\\
        \ref{ex.lists}&
            $\sum_{n\in\nn}\sum_{\ell\in A^n}\yon^{\Pre(\ell)}$&
            $A\in\smset$&
            Queue (discrete, inhomog.)\\
        \ref{ex.balanced_tasks}&
            $\sum_{n\in\nn}\sum_{\ell\in(\nn^{A})^n}\yon^{\Pre(\ell)}$&
            $A\in\smset$&
            Multiset-queue\\
        \ref{ex.potluck}&
            $\sum_{V\ss U}\yon^{\pow(V)}$&
            $U\in\smset$&
            Union then intersection\\
        \ref{ex.comm_list_monoid}&
            $\sum_{n\in\nn}\yon^{\ord{n}}$&
            ---&
            Product of sets, projections\\
        \ref{ex.coprod}&
            $\sum_{c\in\Ob(\cat{C})}\yon^{\cat{C}[c]}$&
            $\cat{C}\in\smcat$ has coprods&
            U.\ property of coproduct\\
        \ref{ex.cart_closed}&
            $\sum_{c\in\Ob(\cat{C})}\yon^{\cat{C}[c]}$&
            $\cat{C}\in\smcat$ cart.\ closed&
            U.\ property of exponential\\
        \ref{ex.presheaf_coprod}&
            $\sum_{U\in\Cat{Open}(X)}F(U)$&
            $X$ space, $F$ sheaf&
            Union and restriction\\
        \ref{ex.simplices}&
            $\sum_{n\in\Delta_+}\yon^{\Delta_+\op[n]}$&
            Algebraist's simplex&
            Pullback inclusions\\
        \ref{ex.oplax}&
            $\sum_{c\in\Ob(\cat{C})}\yon^{F(c)}$&
            $(\cat{C},\odot)\To{F}(\smset,\times)$ oplax&
            Oplax monoidal structure\\
        \ref{prop.operad}&
            $\sum_{n\in\nn}\cat{O}_n\yon^{\ord{n}}$&
            $\cat{O}$ an operad&
            Add and copy-compose\\
        \ref{ex.probabilities}&
            $\sum_{n\in\nn}\sum_{p\in\Dist(\ord{n})}\yon^{\ord{n}}$&
            ---&
            Sample probability spaces\\
        \ref{ex.traj}&
            $\sum_{v:\rr^2\to\rr^2}\yon^{\rr^2}$&
            ---&
            Add vector to point\\
        \ref{sec.parallel}&
            $p\otimes q$&
            $p,q$ carry collectives&
            Run in parallel\\
        \ref{sec.product}&
            $p\times q$&
            $p,q$ carry collectives&
            Both contribute, one returns\\
        \ref{sec.composite}&
            $p\tri q$&
            $p,q$ carry collectives&
            Run in series\\
        \ref{sec.free}&
            $\sum_{n\in\nn}\sum_{c\in p(1)^n}\yon^{\sum_{i\in\ord{n}}p[c_i]}$&
            $p\in\poly$&
            Free collective on $p$
    \end{tabular}
    \caption{Examples of collectives}
    \label{tab:my_label}
\end{table}

As one can imagine from their generality, there is an abundance of collectives.
To explore practical applications, we first build intuitions using simple, canonical examples.
Many of the examples below can be generalized, specialized, or otherwise modified in interesting ways: we encourage the reader to explore these variations to come up with additional examples on their own.

\begin{example}[Monoids as donation boxes] \label{ex.monoid_linear}
    Any monoid $(M,e,*)$ forms a collective with interface $M\yon^\1$ and a trivial distribution: elements of $M$ are possible contributions, and $*\colon M\times M \to M$ is the aggregation (with unit $e\in M$).
    For each contribution, there is only one possible return (let's imagine it's the singleton set $\1\cong\{\text{go-team!}\}$), so every distribution is the unique map $\1\to\1\times\1$.
    According to our notation, we would write this distribution function as
    \[
        \frac{m}{m*n}(\text{go-team!})=\text{go-team!}=\frac{n}{m*n}(\text{go-team!})
    \]
    for every $m,n\in M$.
    This is a commutative collective if and only if the monoid itself is commutative.
    
    We can think of contributing to this collective as \emph{making a donation} without expecting a return: the donors' contributions are aggregated into a single donation, and every contributor just receives ``go-team!'' regardless of what they contribute. The donations could be monetary or otherwise: people could donate their thoughts in a survey, with the aggregate contribution being a list of answers to a question.
\end{example}

\begin{example}[Sets as distribution lists]
\label{ex.set_repr}
    For any set $S$ there is a collective with interface $\1\yon^S$ that requires nothing for contributions---i.e.\ $C=\1\cong\{\text{present}\}$ is the trivial monoid---and simply distributes messages from the set $R[\text{present}]=S$ as its only function.
    That is, if there are $n$ members of the collective, the distribution function $S\to S\times\cdots\times S$ sends $s\mapsto(s,s,\ldots,s)$. According to our notation, we would write this distribution function as
    \[
        \frac{\text{present}}{\text{present}*\text{present}}(s)=s
    \]
    This collective is commutative.

    We can interpret this collective as a simple \emph{distribution list}: the elements of $S$ can be thought of as messages which are copied and distributed to everyone present.
\end{example}

A simple example of a collective with nontrivial aggregation and distribution is the \emph{stakeholder collective}. This encapsulates the basic idea of \emph{adding up} the contributions and \emph{dividing up} the returns.

\begin{example}[Stakeholders]
\label{ex.investment_reals}

Let $\rr_+$ be the set of positive real numbers, and $\rr_{\geq 0} = \{0\}\cup\rr_+$ be nonnegatives, representing a real-world quantity such as money.

The interface of the \emph{stakeholder} collective is $\{0\}\yon^{\{0\}}+\rr_+\yon^{\rr_+}$.
The aggregation is addition, and the distribution $\rr_{\geq 0} \to \rr_{\geq 0}\times \rr_{\geq 0}$ takes a return and divides it proportionally:
\[
\frac{a}{a+b}(t):=\frac{at}{a+b}
\qqand 
\frac{b}{a+b}:=\frac{bt}{a+b}
\]
where we define $\frac{0}{0}$ to be $0$. Here the fraction notation works very well---and makes it easy to check that this collective is commutative.

We can interpret this collective as follows. Given $n$ stakeholders, each stakeholder $i \in \ord{n}$ contributes an amount $a_i \in \rr_+$ of Resource A, for a total aggregate quantity of $\sum_{i \in \ord{n}} a_i$.
Based on this total, an amount $t \in \rr_+$ of Product T is returned, and each stakeholder receives their part $a_i t / \sum_{i \in \ord{n}} a_i$ of Product T proportional to their contribution of Resource A.

When defining this collective, we wanted the set of returns on the contribution $0$ to be the singleton set $\{0\}$ rather than $\rr_+$ in order to properly handle the case where all members contribute nothing.
Note that \eqref{gath.eq1} guarantees that
\[
    \frac{0}{0+0}=\id_{R[0]},
\]
so if two members each contribute $0$, any return $t\in R[0]$ on their aggregate contribution would be duplicated for both members to receive!
The only way this could make sense would be for $0$ to be the only possible return on a contribution of $0$, and this still plays nicely with cases where some members contribute $0$ and other members contribute positive quantities.
\end{example}

So far, apart from the case of noncommutative monoids, all of our collectives have been commutative, so that the order of contributions does not matter. Yet there are many natural situations which require ordered protocols: reserving a spot, or consolidating ordered data.
The following is an example of a collective that is not commutative, even though its underlying monoid is.

\begin{example}[First come, first served reservations]
\label{ex.ylist_monoid}
We consider another noncommutative collective, this time with interface
\[\sum_{x\in R}\yon^{[0,x]}\]
where $R=\rr_{\geq 0}$ is the set of nonnegative reals and $[a,b]$ denotes the closed interval. We think of $x\in R$ as an amount of time to be reserved by a member of the collective. 

The aggregation is the addition map $(m, n) \mapsto m+n$ (with unit $0 \in R$), which is commutative, satisfying \eqref{eqn.comm_agg}, while the distribution is given by
\[
    \frac{m}{m+n}(d) = \min(d,m) \qqand \frac{n}{m+n}(d) = \max(0,d-m),
\]
which does not satisfy \eqref{eqn.comm_dis}, making this collective noncommutative.

To illustrate the behavior of this collective, imagine that when members contribute $m_1,\dots,m_k$, they are each submitting a request for an amount of time with an important visitor. The total time requested is $m_1+\cdots+m_k$. If the visitor can only spend $d$ hours with the group, then they will spend $m_1$ hours with the first member, $m_2$ with the second, etc., until the $d$-many hours run out. 

Alternatively, one can imagine the members drawing successively higher marks on a beaker; when water is poured, it is distributed to the lower members first.

This whole example works with $R=\nn$ instead of $\rr_{\geq 0}$ as well; in fact, the $R=\nn$ case is a special case of the next example, when $A=\1$.
\end{example}

\begin{example}[First come, first served task scheduler]\label{ex.lists}
Here is another example of a noncommutative collective.
Given a set $A$, there is a collective whose contributions are (finite, possibly empty) lists of elements in $A$ and whose returns on each list $\ell$ are the prefixes of that list.\footnote{A \emph{prefix} of a list $[\ell_1,\ell_2,\ldots,\ell_n]$ is the list of the first $i$ terms $[\ell_1,\ldots,\ell_i]$ for some $0\leq i\leq n$. We denote the set of prefixes of a list $\ell$ by $\Pre(\ell)$.} The aggregation is the concatenation map, which we denote by $\concat$, whose unit is the empty list, which we denote by $[]$.
The distribution can be described as follows.

Given lists $k\coloneqq[k_1,\ldots,k_m]$ and $\ell\coloneqq[\ell_1,\ldots,\ell_n]$, a prefix of $k\concat\ell$ either has the form $[k_1,\ldots,k_i]$ for some $i\leq m$ or $k\concat[\ell_1,\ldots,\ell_j]$ for some $j\geq1$.
In the first case, we define
\[
    \frac{k}{k\concat\ell}([k_1,\ldots,k_i]) = [k_1,\ldots,k_i] \qqand \frac{\ell}{k\concat\ell}([k_1,\ldots,k_i]) = [],
\]
whereas in the second case, we define
\[
    \frac{k}{k\concat\ell}(k\concat[\ell_1,\ldots,\ell_j]) = k \qqand \frac{\ell}{k\concat\ell}(k\concat[\ell_1,\ldots,\ell_j]) = [\ell_1,\ldots,\ell_j].
\]
We think of this as a \emph{task-scheduling} collective. Each member contributes a list of tasks for the collective, elements of some fixed set $A$. The task-lists are aggregated to make one long list of tasks, concatenated in the order of contributions. Then, the first few of those tasks are completed, and everyone receives the tasks on their own list that were completed.
\end{example}

One might have the idea that this is not very fair: the first member to make a contribution gets too much advantage. Here is a version that takes a more balanced approach.

\begin{example}[Balanced task scheduler]\label{ex.balanced_tasks}
Again fix a set of tasks $A$, and consider the set $\nn^A$ of multisets of $A$.%
\footnote{A multiset $f\colon A\to\nn$ is like subset of $A$, except that the same element $a\in A$ can be chosen multiple times: specifically, each element $a$ appears $f(a)$-many times.}
We denote the empty multiset by $0$ and the union of two multisets $k_1,k_2$ by $k_1+ k_2$; it is given by $(k_1+k_2)(a)\coloneqq k_1(a)+k_2(a)$. 
Consider the collective for which a contribution is a list of multisets in $A$, and whose returns at such a list is the set of its prefixes, as in \cref{ex.lists}. We consider a list to be equivalent to any list that can be obtained from it by adding $0$'s at the end.

The aggregation function takes two lists $[k_1,\ldots,k_i]$ and $[\ell_1,\ldots,\ell_j]$, replaces them with two lists of the same length $n=\max(i,j)$ by taking the missing multisets to be $0$, and produces
\[
[k_1+\ell_1,k_2+\ell_2,\ldots,k_n+\ell_n].
\]
A return on this aggregated list of multisets is just one of its prefixes, which we distribute to each contributing member as a prefix of their own contributed list of the same length.
So the distribution function maps a return
\[
    [k_1+\ell_1,k_2+\ell_2,\ldots,k_m+\ell_m]
\]
for some $0\leq m\leq n$ to the returns
\[
    ([k_1,k_2,\ldots,k_m],[\ell_1,\ell_2,\ldots,\ell_m]).
\]

The semantics of this collective are as follows. Each member contributes a list of bundles (multisets) of tasks, ordered by priority, so that the tasks in the first bundle of the list are of the highest priority, the tasks in the second bundle of the second highest priority, and so on. Contributions are aggregated fairly in the sense that all of the highest-priority tasks in $k$ and all of the highest-priority tasks in $\ell$ together form the bundle of highest-priority tasks in $k+\ell$. A return on a list of task bundles is the completion of zero or more of these bundles; the collective always finishes any bundle they start (so the elements of $A$ had better be eminently doable tasks!). Then each completed task bundle is distributed to the member that completed it.

Note that this collective is commutative, and it is precisely this commutativity that ensures this collective (unlike the previous example) gives no comparative advantages or disadvantages based on the members' contribution order.
\end{example}

\begin{example}[Potluck planner]\label{ex.potluck}
Given a set $U$, we give a commutative collective with interface
\[
    \sum_{V\ss U}\yon^{\pow(V)},
\]
where $\pow(V)$ is the power set of $V$.
So the possible contributions are the subsets of $U$, and the possible returns on each contribution are in turn the subsets of that contribution.

The aggregation is the union $(V, W) \mapsto V \cup W$ (with unit $\varnothing\ss U$), while the distribution is given by
\[
    \frac{V}{V\cup W}(X) = V \cap X \qqand \frac{W}{V\cup W}(X) = W \cap X.
\]

We can interpret this collective as a \emph{potluck planner}, as follows. Each member contributes a collection of dishes that they could bring. The collections are aggregated into a menu, from which only a subset of desired dishes are chosen. This subset is distributed back to the members so that each member is instructed to bring every desired dish that they offered to make.

Note that if we follow this protocol, we may end up asking multiple members to bring the same dish---in fact, every member who offered to bring a desired dish will be asked to bring that dish.
If we wanted to avoid this scenario, we could instead define the distribution to be
\[
    \frac{V}{V\cup W}(X) = V \cap X \qqand \frac{W}{V\cup W}(X) = (W \cap X) \setminus V,
\]
so that the \emph{first} member to offer a desired dish is always the one asked to bring it.
Or perhaps we want to incentivize members to sign up early, in which case we could define the distribution in yet another way:
\[
    \frac{V}{V\cup W}(X) = (V \cap X)\setminus W \qqand \frac{W}{V\cup W}(X) = V\cap X,
\]
so that the \emph{last} member to offer a desired dish is always the one asked to bring it.
These two alternatives are still collectives, but they are no longer commutative: although the aggregation is unchanged and thus still satisfies \eqref{eqn.comm_agg}, the distribution no longer satisfies \eqref{eqn.comm_dis}.
We have sacrificed commutativity to eliminate redundancy.
\end{example}

\begin{example}[Single-question survey]
\label{ex.comm_list_monoid}

We consider a collective with interface
\[\sum_{n\in\nn}\yon^{\ord{n}}=\yon^\0+\yon^\1+\yon^\2+\cdots.\]
A contribution is a natural number, and the aggregation is multiplication, with unit $1 \in \nn$. The possible returns on a contribution $n\in\nn$ are the natural numbers less than $n$.
The left distribution $\frac{m}{m\cdot n}$ sends $i\mapsto i\bmod m$, while the right distribution $\frac{n}{m\cdot n}$ sends $i\mapsto i\bdiv m$.

If $n$ members are contributing, we can think of their aggregate contribution as an $n$-dimensional grid: each member contributes the size of their own dimension. Let's suppose $n=2$ and we have two members contributing $7$ and $4$:
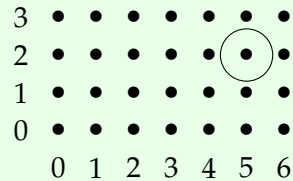
\[
\begin{tikzpicture}[scale=.5]
    \foreach \i in {0,...,6} {
        \foreach \j in {0,...,3} {
        \node at (\i,\j) (a\i\j) {$\bullet$};
        }
    }
    \node[circle, inner sep=.4pt, draw, fit=(a52)] {};
    \foreach \i in {0,...,6} {
        \node at (\i, -1) {\i};
    }
    \foreach \j in {0,...,3} {
        \node at (-1, \j) {\j};
    }
\end{tikzpicture}
\]
The resulting returns are integers between $0$ and $27$, distributed by projection. For example, the circled node represents $19$ and is distributed as $19\bmod 7=5$ and $19\bdiv 7=2$.

We can interpret this collective as follows. Given $n$ interviewers, each interviewer $i\in\ord{n}$ contributes a question with exactly $a_i\in\nn$ possible answers to choose from.
These questions are aggregated into a single question with $\prod_{i\in\ord{n}}a_i$ possible answers, each corresponding to a combination of answers for the questions that were originally contributed.
When that single question is answered, the results for each original question are distributed back to the interviewer who asked that question.

Perhaps surprisingly, this collective is not commutative.
This is because the way the projection map is defined differs depending on which contribution is made first.

This example may not be the most satisfying way to construct a survey: these aggregate single-question surveys will quickly become very unwieldy to answer!
In \cref{ex.multi-question}, we will see a natural way to turn our single-question survey collective into a multi-question survey collective.
\end{example}

The following examples are more abstract, intended for those familiar with category theory.

\begin{example}[Categories with coproducts] \label{ex.coprod}
Let $\cat{C}$ be a category with finite coproducts given by $+$.
For each $c\in\Ob\cat{C}$, define
\begin{equation}\label{eqn.not_C[c]}
    \cat{C}[c]\coloneqq\coprod_{c'\in\Ob\cat{C}}\Hom_{\cat{C}}(c,c')
\end{equation}
to be the set of all morphisms in $\cat{C}$ with domain $c$.
Then there is a commutative collective with interface
\[
    \sum_{c\in\Ob\cat{C}}\yon^{\cat{C}[c]}.
\]
Here the aggregation is the coproduct operation $(c,d)\mapsto c+d$ (whose unit is the initial object).
To ensure that the aggregation is well-defined, we fix a representative for each isomorphism class of $\cat{C}$ that the coproduct operation will pick out.
Meanwhile, the distribution sends every morphism $f\colon c+d\to e$ in $\cat{C}[c+d]$ to its restrictions along $c$ and $d$:
\[
    \frac{c}{c+d}(f)=\iota_c\then f \qqand \frac{d}{c+d}(f)=\iota_d\then f,
\]
where $\iota_c\colon c\to c+d$ and $\iota_d\colon d\to c+d$ are the canonical coproduct inclusions. 
\end{example}

\begin{example}[Cartesian closed categories]\label{ex.cart_closed}
Let $\cat{C}$ be a cartesian closed category, and for each $c\in\Ob\cat{C}$, define $\cat{C}[c]$ as in \eqref{eqn.not_C[c]}.
Then there is a commutative collective with the same interface
\[
    \sum_{c\in\Ob\cat{C}}\yon^{\cat{C}[c]}
\]
as in \cref{ex.coprod}, but where the aggregation is the product operation $(c,d)\mapsto c\times d$ (whose unit is the terminal object).
Again, we fix a representative for each isomorphism class of $\cat{C}$ that the product operation will pick out.
Meanwhile the distribution $\cat{C}[c\times d]\to\cat{C}[c]\times\cat{C}[d]$ sends every morphism $f\colon c\times d\to e$ to its curried forms, $c\to e^d$ and $d\to e^c$.
To ensure that the distribution is well-defined, we also fix a representative for each isomorphism class of $\cat{C}$ that the exponentiation operation will pick out.
\end{example}

\begin{example}[Presheaves on topological spaces]\label{ex.presheaf_coprod}
Let $X$ be a topological space with open sets $\mathcal{O}(X)$, and let $F$ be a presheaf of sets on $X$.
Then there is a commutative collective with interface
\[
    \sum_{U\in\mathcal{O}(X)}\yon^{F(U)}
\]
where the aggregation is the union $(U,V)\mapsto U\cup V$ of open sets (whose unit is the empty set), while the distribution $F(U\cup V)\to F(U)\times F(V)$ is given by the restrictions:
\[
    \frac{U}{U\cup V}=\res_{U}^{U\cup V}\qqand\frac{V}{U\cup V}=\res_{V}^{U\cup V}.
\]

We can replace $\mathcal{O}(X)$ with an arbitrary category with finite coproducts (again, we will choose a fixed representative in the isomorphism class of every coproduct in $\cat{C}$ for the aggregation to pick out).
In fact, \cref{ex.coprod} is a special case of this generalization.
We will further generalize this example in \cref{ex.oplax}.
\end{example}

\begin{example}[Simplices]\label{ex.simplices}
Let $\Delta_+$ denote the category of finite ordered sets and order-preserving maps. Its objects form the contributions of a collective, with aggregation $+$, and the returns on each object are its incoming morphisms.

That is, consider the set $C\coloneqq\nn$ of objects $n\in\Delta_+$, which we call \emph{simplices}, and for each one, consider the set \[R[n]\coloneqq\sum_{m\in\nn}\Big\{f\colon \{1,\ldots,m\}\to\{1,\ldots,n\}\;\Big|\; i\leq j\implies f(i)\leq f(j)\Big\}\]
of incoming morphisms; we refer to elements of $R[n]$ as the \emph{faces} of simplex $n$. Then define the aggregation map to be ordinary sum, $(n,n')\mapsto n+n'$, with unit $0\in C$.
A return on $n+n'$ is a natural number $m\in\nn$ and an order-preserving map $g\colon \{1,\ldots,m\}\to\{1,\ldots,n+n'\}$. It is distributed as
\[
\frac{n}{n+n'}(g)=f
\qqand
\frac{n'}{n+n'}(g)=f'
\]
where $f$ and $f'$ are given by taking pullbacks
\[
\begin{tikzcd}
\bullet\ar[r]\ar[d, "f"']\ar[dr,phantom, near start, "\lrcorner"]&\{1,\ldots,m\}\ar[d, "g"]&
\bullet\ar[l]\ar[d,"f'"]\ar[dl,phantom, near start, "\llcorner"]\\
\{1,\ldots,n\}\ar[r]&\{1,\ldots,n,n{+}1,\ldots,n{+}n'\}\ar[from=r]&\{n{+}1,\ldots,n{+}n'\}
\end{tikzcd}
\]

From the point of view of simplicial sets, each member contributes a simplex (the convex hull of some number of points). These are aggregated by taking the convex hull of the union of points. A return is simply a face of that simplex, and it is distributed to the contributing members based on which of its points each contains. 

This looks like it could be a special case of \cref{ex.presheaf_coprod}, but it is not; indeed $+$ is not a coproduct in $\Delta_+$.
But in fact, both \cref{ex.presheaf_coprod} and this example are special cases of the generalization below.
\end{example}

\begin{example}[Oplax monoidal functors to the cartesian category of sets]\label{ex.oplax}
Let $\cat{C}$ be a monoidal category with unit $I$ and monoidal product $\odot$.
Then given an oplax monoidal functor $F\colon(\cat{C},I,\odot)\to(\smset,\ord{1},\times)$, there is a collective with interface
\[
    \sum_{c\in\Ob\cat{C}}\yon^{F(c)}.
\]
We fix a representative for each isomorphism class of $\cat{C}$ and say that the aggregation sends $(c,d)$ to the representative of the isomorphism class containing $c\otimes d$.
Meanwhile, we define the distribution to be the morphism $F(c\odot d)\to F(c)\times F(d)$ given by the oplax monoidal functor.
\end{example}

The morphisms in $\cat{C}$ are not being used in \cref{ex.oplax}. Indeed, \emph{every} collective is an example of \cref{ex.oplax} with $\cat{C}$ discrete, as we observe below.

\begin{proposition}
A collective is an oplax monoidal functor from a discrete monoidal category to the cartesian category of sets.
\end{proposition}
\begin{proof}
A discrete monoidal category equipped with an oplax monoidal functor $R$ to the cartesian category of sets is precisely a monoid $(C,1,*)$ equipped with a set $R[c]$ for each $c\in C$ and a function $R[a*b]\to R[a]\times R[b]$ for each $a,b\in C$.
The associativity and unitality conditions for an oplax monoidal functor correspond precisely to the cancellation equations for a collective.
\end{proof}

Here is another way to construct a collective from another ubiquitous categorical construction: an operad.

\begin{proposition}\label{prop.operad}
Let $\cat{O}$ be an operad (with one object), so that $\cat{O}_n$ denotes the set of $n$-ary operations for any $n\in\nn$. Then the polynomial $\sum_{n\in\nn}\cat{O}_n\yon^{n}$ carries a collective structure.
\end{proposition}
\begin{proof}
For experts, this comes from the lax monoidal functor $(\poly,\yon,\otimes)\to(\poly,\yon,\tri)$, whose laxators are the identity $\yon\to\yon$ and the natural map $p\otimes q\to p\tri q$ (see \cite{poly_book}). Indeed, an operad is a special case of a cartesian polynomial monad, and we simply compose $p\tri p\to p$ and $\yon\to p$ with the laxators.
\end{proof}

Let's come back down to earth by explaining the idea and giving an example. The idea is that an operad $\cat{O}$ consists of a set $\cat{O}_n$ of $n$-ary operations for each $n$, which is closed under composition. We form the corresponding collective by defining a contribution to be a choice of $n\in\nn$ and an $n$-ary operation $f\in\cat{O}_n$. Say $e$ is $m$-ary and $f$ is $n$-ary. Then we get an $(m*n)$-ary operation to serve as their aggregate by
\[e*f\coloneqq e\circ(f,f,\ldots,f).\]
A return at $(n,f)$ is simply a number $1\leq i\leq n$. So given a return on $m*n$, we simply map it to a return on each factor via the projections
\[\{1,\ldots,m\}\from\{1,\ldots,m\}\times \{1,\ldots,n\}\to\{1,\ldots,n\}.\]
An example should further clarify.

\begin{example}[Probabilistic events]\label{ex.probabilities}
There is an operad of probabilistic events, from which we will get a collective. A probabilistic event is a pair $(n,p)$ where $n\in\nn$ and \[
p\in\{(a_1,\ldots,a_n)\in\rr^n_{\geq 0}\mid a_1+\cdots+a_n=1\}
\]
We'll denote $(n,p)$ by $(n:p)$ to mean ``$n$ events of probabilities $p_1,\ldots,p_n$.

For example an unfair coin-flip and a fair die-roll are both probabilistic events, the first one given by $(2:\frac{1}{3},\frac{2}{3})$ and the second given by $(6:\frac{1}{6},\frac{1}{6},\frac{1}{6},\frac{1}{6},\frac{1}{6},\frac{1}{6})$. Imagine each member of the collective contributes a probabilistic event: Alice wants to know what side will come up on her crazy coin, and Bob wants to know what number will come up on his fair die. We aggregate their questions according to the operad: we first flip the coin, and then we roll the die; this results in a probabilistic event of type $(12:\frac{1}{18},\ldots,\frac{1}{18},\frac{1}{9},\ldots,\frac{1}{9})$.

A return on $(n:p)$ is just a number $(1,\ldots,n)$, i.e. the return of type coin-flip is the answer of what was flipped! A result of the coin flip+die roll is distributed to Alice and Bob by returning the result of the coin flip to Alice and the result of the die roll to Bob.
\end{example}

Finally, here is a collective that appears in math of a slightly different flavor.
\begin{example}[Trajectories]\label{ex.traj}
Suppose that for every point $x=(x_1,x_2)\in\rr^2$ we pick a vector $v(x)=(v_1(x),v_2(x))\in\rr^2$, a direction to move. Assuming these vary continuously, the resulting data is called a \emph{vector field} on $\rr^2$. 
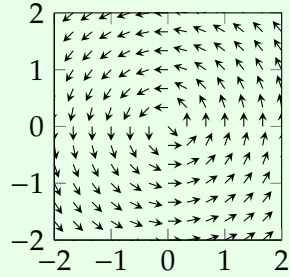
\begin{equation}\label{eqn.vector_field}
    \begin{tikzpicture}
        \begin{axis}[
            xmin = -2, xmax = 2,
            ymin = -2, ymax = 2,
            zmin = 0, zmax = 1,
            axis equal image,
            xtick distance = 1,
            ytick distance = 1,
            view = {0}{90},
            scale = 1.25,
            height=4cm
        ]
            \addplot3[
                point meta = {sqrt(x^2+y^2)},
                quiver = {
                    u = {-y/sqrt(x^2+y^2)},
                    v = {x/sqrt(x^2+y^2)},
                    scale arrows = 0.25,
                },
                -stealth,
                domain = -4:4,
                domain y = -4:4,
            ] {0};
        \end{axis}
    \end{tikzpicture}
\end{equation}
The vector fields on $\rr^2$ form the contributions of a collective, where \eqref{eqn.vector_field} shows one such contribution.%
\footnote{The vector fields on a more general thing, e.g.\ on a manifold $M$, also form a collective.} 
We will write $x+v\coloneqq(x_1+v_1,x_2+v_2)$ to denote the result of jumping out from the point $x$ according to the vector $v$.

The contributions in this collective are the vector fields, so we could imagine each member saying ``I don't know what point we're at exactly, but whatever it is, here's a formula for the direction I want to go.'' The neutral contribution is the 0-vector field, where $v_1(x)=v_2(x)=0$ for all $x\in\rr^2$. We denote the aggregate of contributions $v,w$ as $v\then w$; it is defined as follows
\[
 (v\then w)(x)=x+v(x)+w(x+v(x))
\]
In other words, starting at any point $x$, it chooses the vector given by jumping out from $x$ according to vector field $v$, and then jumping out from there (i.e.\ from the landing point $x+v(x)$) according to vector field $w$.

A return on a vector field $v$ is just a point in $\rr^2$; in other words, the return type does not depend on the contribution. However, the distribution function does. The formula is easiest to appreciate on more than just two contrubtions, so suppose given vector fields $t,u,v,w$, aggregated to $t\then u\then v\then w$, meaning ``jump by $t$ then by $u$ then by $v$ then by $w$''. A return on that aggregate is just a point $x_0\in\rr^2$, but it is distributed by
\begin{align*}
\frac{t}{t\then u\then v\then w}(x_0)&=x_0,\\
\frac{u}{t\then u\then v\then w}(x_0)&=x_0+t(x_0),\qquad\text{call it }x_1\\
\frac{v}{t\then u\then v\then w}(x_0)&=x_1+u(x_1),\qquad\text{call it }x_2\\
\frac{w}{t\then u\then v\then w}(x_0)&=x_2+v(x_2),\qquad\text{call it }x_3
\end{align*}
In other words, a point $x_0\in\rr^2$ is distributed to the contributing members $t,u,v,w$ by doing exactly what was said: jumping out from $x_0$ according to their vector fields in that order. It returns $x_0$ to $t$, jumps out by $t(x_0)$ and return the result $x_1$ to $u$, jumps out by $u(x_1)$ and returns the result $x_2$ to $v$, and jumps out by $v(x_2)$ and returns the result $x_3$ to $w$.
\end{example}

\chapter{New collectives from old}
\label{sec:making}

In this section, we present several ways of constructing new collectives from old. Throughout the following, we let $\mathscr{C}$ and $\mathscr{D}$ be two collectives with contribution sets $C$ and $D$, return sets $R[c]$ and $S[d]$ for each $c\in C$ and $d\in D$, neutral elements $\e_C\in C$ and $\e_D\in D$, aggregations $*\colon C\times C\to C$ and $\diamond\colon D\times D\to D$, and the corresponding notations $\dfrac{c}{c*c'}$, $\dfrac{c'}{c*c'}$ and $\dfrac{d}{d\diamond d'}$, $\dfrac{d'}{d\diamond d'}$ for distributions.
Denoting their interaces in polynomial form, we write
\begin{gather*}
    \mathscr{C}:=\left(\sum_{c\in C} \yon^{R[c]},\,\e_C,\,*\right) \qqand 
    \mathscr{D}:=\left(\sum_{d\in D} \yon^{S[d]},\,\e_D,\,\diamond\right)
\end{gather*}

In the following subsections we discuss operations on collectives that are easy to describe in terms of polynomial functors. Since we don't assume our readers know anything about polynomial functors, our goal instead is to explain what these operations do intuitively. However, we quickly explain for experts (see \cite{poly_book}):
\begin{itemize}
    \item \cref{sec.parallel}: the $\otimes$-structure (Dirichlet product) on $\poly$ lifts to $\otimes$-monoids;
    \item \cref{sec.product}: the $\times$-structure (Cartesian product) on $\poly$ lifts to $\otimes$-monoids;
    \item \cref{sec.composite}: the $\tri$-structure (composition product) on $\poly$ lifts to $\otimes$-monoids; and
    \item \cref{sec.free}: the forgetful functor from $\otimes$-monoids to $\poly$ has a left adjoint (existence of free collectives).
\end{itemize}
 Some of these have a more intuitive meaning than others, as we will see.
 
\section{Parallel collectives}\label{sec.parallel}
Given two collectives, one way to combine them is such that ``the two protocols happen in parallel.'' More precisely there is a new collective, which we call the \emph{parallel collective} and denote $\mathscr{C}\otimes\mathscr{D}$, with interface
\[
    \sum_{(c,d)\in C\times D} \yon^{R[c]\times S[d]},
\]
i.e.\ with contribution set $C\times D$ and returns $R[c]\times S[d]$ for each $c,d$. The rest of its structure is as follows:
\begin{itemize}
    \item its neutral contribution is $(\e_C,\e_D)\in C\times D$;
    \item its aggregation sends $((c,d),(c',d'))\mapsto (c*c', d\diamond d')$;
    \item its distribution $R[c*c']\times S[d\diamond d'] \to (R[c]\times S[d])\times (R[c']\times S[d'])$ sends
    \[
        (r,s) \mapsto \left(\left(\frac{c}{c*c'}(r), \frac{d}{d\diamond d'}(s)\right), \left(\frac{c'}{c*c'}(r), \frac{d'}{d\diamond d'}(s)\right)\right).
    \]
\end{itemize}
It behaves like the two original collectives $\mathscr{C}$ and $\mathscr{D}$ acting in parallel, in the following sense:
\begin{itemize}
    \item A contribution of $\mathscr{C}\otimes\mathscr{D}$ is just a pair $(c,d)$ of contributions, one from $\mathscr{C}$ and another from $\mathscr{D}$; to aggregate contributions from a list of members, we just aggregate the $\mathscr{C}$-parts and the $\mathscr{D}$-parts independently.
    \item A return on a contribution $(c,d)$ is just a pair of returns, one from $R[c]$ and one from $S[d]$. Given a list of contributions in $\mathscr{C}\otimes\mathscr{D}$ and a return on their aggregate, we distribute returns to each member by distributing the return on the $\mathscr{C}$-part and the return on the $\mathscr{D}$-part independently.
\end{itemize}

For experts, if $\mathscr{C}=(p,\e_p,*)$ and $\mathscr{D}=(q,\e_q,\diamond)$ are $\otimes$-monoids, the relevant monoidal product on their Dirichlet product $p\otimes q$ follows from the symmetry of $\otimes$: \[(p\otimes q)\otimes (p\otimes q)\cong(p\otimes p)\otimes (q\otimes q)\To{*\otimes\diamond} p\otimes q,\]
and similarly for monoidal unit
\[
\yon\cong\yon\otimes\yon\To{\e_p\otimes\e_q}p\otimes q.
\]

\section{Product collectives}\label{sec.product}

Another way to combine two collectives is such that ``the contributions happen in parallel, but only one return is distributed.'' More precisely there is a new collective, which we call the \emph{product collective} and denote $\mathscr{C}\times\mathscr{D}$, with interface
\[
    \sum_{(c,d)\in C\times D} \yon^{R[c]+S[d]},
\]
i.e.\ with contribution set $C\times D$ and returns $R[c]+ S[d]$ for each $c,d$.\footnote{Here $R[c]+S[d]$ denotes the disjoint union of sets $R[c]$ and $S[d]$.} The rest of its structure is as follows:
\begin{itemize}
    \item its neutral contribution is $(\e_C,\e_D)\in C\times D$;
    \item its aggregation sends $((c,d),(c',d'))\mapsto (c*c', d\diamond d')$;
    \item its distribution $R[c*c']+ S[d\diamond d'] \to (R[c]+ S[d])\times (R[c']+ S[d'])$ works by cases, either $r\in R[c*c']$ or $s\in S[d\diamond d']$:
    \begin{align*}
        r &\mapsto \left(\frac{c}{c*c'}(r), \frac{c'}{c*c'}(r)\right)\\
        s &\mapsto \left(\frac{d}{d\diamond d'}(s),
        \frac{d'}{d\diamond d'}(s)\right).
    \end{align*}
\end{itemize}
On each round, each contributing member produces a contribution of both sorts, but altogether they only receive a return from one or the other of them.

For experts, if $\mathscr{C}=(p,\e_p,*)$ and $\mathscr{D}=(q,\e_q,\diamond)$ are $\otimes$-monoids, the relevant monoidal product on their Cartesian product $p\times q$ follows from the universal property of products: \[(p\times q)\otimes (p\times q)\to(p\otimes p)\times (q\otimes q)\To{*\times \diamond} p\times q,\]
and the monoidal unit is given by the diagonal
\[\yon\to\yon\times \yon\To{\e_p\times\e_q} p\times q.\]

\section{Composite collectives}\label{sec.composite}

Given two collectives, a third---and much more intricate---way to combine them is such that ``the two protocols happen in series.'' More precisely there is a new collective, which we call the \emph{composite collective} and denote $\mathscr{C}\tri\mathscr{D}$, with interface
\[
    \sum_{c\in C}\;\sum_{f\colon R[c]\to D} \yon^{\sum_{r\in R[c]} S[f(r)]},
\]
i.e.\ for which a contribution is a pair $(c,f)$ where $c$ is a $\mathscr{C}$-contribution and $f\colon R[c]\to D$ is a function that takes any $\mathscr{C}$-return $r$ and produces a $\mathscr{D}$-contribution $f(r)\in D$, and for which a return is a pair $(r,s)$, where $r\in R[c]$ and $s\in S[f(r)]$ are returns. The rest of its structure is as follows:
\begin{itemize}
    \item its neutral contribution is $(\e_C,r\mapsto \e_D)\in \sum_{c\in C}\smset(R[c], D)$;
    \item its aggregation sends \[\big((c,f),(c',f')\big)\mapsto \left(\left(c*c'\right), r\mapsto \left(f\left(\frac{c}{c*c'}(r)\right)\diamond f'\left(\frac{c'}{c*c'}(r)\right)\right)\right);\]
    \item its distribution sends $(r,s)$, where $r\in R[c*c']$ and $s\in S[f(\frac{c}{c*c'}(r))\diamond f'(\frac{c'}{c*c'}(r))]$, to
    \[
        (r,s)\mapsto
        \left(
            \left(
                \frac{c}{c*c'}(r),
                \frac{c'}{c*c'}(r)
            \right),
            \left(
            \tilde{s},\tilde{s}'
            \right)
        \right)
    \]
    where
    \[
    \tilde{s}=\frac{f\left(\frac{c}{c*c'}(r)\right)}{f\left(\frac{c}{c*c'}(r)\right)\diamond f'\left(\frac{c'}{c*c'}(r)\right)}
    \qqand
    \tilde{s}=\frac{f'\left(\frac{c'}{c*c'}(r)\right)}{f\left(\frac{c}{c*c'}(r)\right)\diamond f'\left(\frac{c'}{c*c'}(r)\right)}
    \]
\end{itemize}
This looks complicated, but intuitively it behaves like the two original collectives $\mathscr{C}$ and $\mathscr{D}$ acting in series. That is,
\begin{itemize}
    \item A contribution to the composite is a $\mathscr{C}$-contribution and a strategy for taking a $\mathscr{C}$-return on that contribution and producing a $\mathscr{D}$-contribution. When a list of members produce these $(c,f)$ pairs, they are aggregated by first aggregating the $c$'s to form some $c'$ and second by giving the following strategy: for any return $r$ on $c'$, we will distribute it to all the members, see what $\mathscr{D}$-contribution each of their strategies says to produce, and aggregate these.
    \item A return on $(c,f)$ is just a pair $(r,s)$ where $r\in R[c]$ is a $\mathscr{C}$-return on $c\in C$ and $s\in S[f(r)]$ is a return on the subsequent contribution $f(r)\in D$. To distribute the return $(r,s)$ on an aggregate contribution, distribute $r$ to the members using $\mathscr{C}$'s distribution function, and then distribute $s$ to the members using $\mathscr{D}$'s distribution.
\end{itemize}

For experts, if $\mathscr{C}=(p,\e_p,*)$ and $\mathscr{D}=(q,\e_q,\diamond)$ are $\otimes$-monoids, the relevant monoidal product on their composite $p\tri q$ follows from the duoidality of $\tri$ and $\otimes$:
\begin{equation}\label{eqn.duoidal_tensor}
    (p\tri q)\otimes (p\tri q)\to(p\otimes p)\tri (q\otimes q)\To{*\tri\diamond} p\tri q,
\end{equation}
and similarly for monoidal unit
\begin{equation}\label{eqn.duoidal_unit}
\yon\cong\yon\tri\yon\To{\e_p\tri\e_q}p\tri q.
\end{equation}
\section{Free collectives}\label{sec.free}

Suppose $C\in\smset$ is an arbitrary set and similarly that $R[c]\in\smset$ is an arbitrary set for each $c\in C$; we think of each $c\in C$ as an ``atomic'' contribution and each $r\in R[c]$ as an ``atomic'' return.
We write this in polynomial form as $p:=\sum_{c\in C}\yon^{R[c]}$.
While $p$ itself need not be the interface for any collective, we can construct the \emph{free} collective $p^*$ on $p$. Its contribution set 
\[C^*:=\sum_{n\in\nn}C^n
\]
has as its elements all lists $(c_1,c_2,\ldots,c_n)$ of arbitrary finite length in $C$. For each such list, a return is again a list---of the same length---in the corresponding return sets:
\[
R^*[(c_1,c_2,\ldots,c_n)]:=R[c_1]\times R[c_2]\times\cdots\times R[c_n].
\]
The neutral element is the empty list $()$, whose set of returns is the empty product $R^*[()]=1$. The aggregation operation $\plpl\colon C^*\times C^*\to C^*$ is just list concatenation:
\[
(c_1,\cdots,c_n)\plpl(c'_1,\ldots,c'_{n'})\coloneqq(c_1,\cdots,c_m,c'_1,\ldots,c'_{n'}).
\]
Say we have lists $\ell,\ell'$ of lengths $n$ and $n'$, with aggregate $\ell\plpl\ell'$ of length $n+n'$. The distribution subdivides a length $n+n'$ list of returns into two lists, one of length $n$ and one of length $n'$:
\begin{equation}\label{eqn.free_list_distribution}
\frac{\ell}{\ell\plpl\ell'}(r_1,\cdots,r_{n+n'})=(r_1,\cdots,r_n)
\qand
\frac{\ell'}{\ell\plpl\ell'}(r_1,\cdots,r_{n+n'})=(r_{n+1},\cdots,r_{n+n'}).
\end{equation}
This completes the description of the free collective. 

For experts, we have just described the left adjoint to the forgetful functor from collectives to polynomials.

Here is an interpretation of the semantics of the free collective on a polynomial $p=\sum_{c\in C}\yon^{R[c]}$. For each atomic contribution $c\in C$, think of $R[c]$ as ``what the members want to know or hear about'' when they give that contribution.

The free collective on $p$ has as its contributions all lists of atomic contributions. So each member of the collective contributes a list of atomic contributions, e.g. Bob contributes an banana and two paperclips $(b,c,c)$. The contributions from all the contributing members are aggregated by simply concatenating their lists into one big list.

A return is just a list---of exactly the same length---with one return for each atomic contribution. Suppose the collective has decided that members want to see someone's face when they offer a banana and a picture of a paper stack when they offer a paperclip. Then out of the whole list of returns for the big list, Bob will receive a list $(f,s_1,s_2)$ of a facial expression and two pictures of paper stacks.

\begin{example}[Multi-question survey]\label{ex.multi-question}
As promised, free collectives will allow us to turn our single-question survey from \cref{ex.comm_list_monoid} into a multi-question survey.

If we take $p$ to be the interface of our single-question survey collective
\[
    p\coloneqq\sum_{n\in\nn}\yon^{\ord{n}}=\yon^\0+\yon^\1+\yon^\2+\cdots,
\]
then the free collective on $p$ has interface
\[
    \sum_{k\in\nn}\,\sum_{n\in\nn^k}\yon^{\ord{n}_1\times\cdots\times \ord{n}_k}.
\]
A contribution is a finite list of natural numbers, and the aggregation is concatenation, with the empty list $()$ as the unit.
The possible returns on a contribution $(n_1,\ldots,n_k)$, with each $n_j\in\nn$, are lists $(m_1,\ldots,m_k)$ with each $m_j\in\nn$ satisfying $m_j<n_j$.
The distribution is then given by list subdivision according to the lengths of the original contributions, as in \eqref{eqn.free_list_distribution}.

We can interpret this collective as follows.
Given $n$ interviewers, the $i^\text{th}$ interviewer contributes a survey with $k_i\in\nn$ questions, such that the $j^\text{th}$ question has $n_{i,j}\in\nn$ possible answers to choose from.
These questions are aggregated in the order given into a single survey with $\sum_{i\in\ord{n}} k_i$ questions total; or, equivalently, an $n$-part survey, where the $i^\text{th}$ part has the $k_i$ questions contributed by the $i^\text{th}$ interviewer.
Then the $j^\text{th}$ question in the $i^\text{th}$ part has $n_{i,j}\in\nn$ possible answers to choose from, labeled with the natural numbers less than $n_{i,j}$.

When the survey is filled out, an answer $m_{i,j}\in\nn$ with $m_{i,j}<n_{i,j}$ is chosen for the $j^\text{th}$ question in the $i^\text{th}$ part.
Then the answers for all the questions in the $i^\text{th}$ part are distributed back to the $i^\text{th}$ interviewer.
\end{example}

\chapter{Future Directions}
\label{sec:directions}

There are two main avenues for future work on collectives: developing the theory and giving examples. On the theory side, we hope to further investigate the properties of the category of collectives, i.e. the category of $\otimes$-monoids in $\poly$ and the monoid morphisms between them. Moreover, we wish to give intuitive interpretations of these collective morphisms, as well as categorical concepts such as general limits and colimits of collectives.

There are several natural ways to generalize collectives that may be of interest.
A collective can be viewed as a $\poly$-enriched category with one object; so a natural generalization is to consider $\poly$-enriched categories with multiple objects.
We could also consider (left or right) modules or bimodules over collectives or bimodules over collectives.
These may have intuitive interpretations as well.

We can also equip collectives with other structures and investigate their categorical properties and intuitive interpretations.
While a $\otimes$-monoid in $\poly$ is a collective, a $\otimes$-comonoid in $\poly$ is just a set of monoids.
But a $\otimes$-bimonoid is a collective that comes equipped with a compatible $\otimes$-comonoid structure.
From the duoidality of $\otimes$ and $\tri$ given in $\eqref{eqn.duoidal_tensor}$ and $\eqref{eqn.duoidal_unit}$, we can also consider $(\otimes, \tri)$-bimonoids, which are $\otimes$-monoids equipped with a compatible $\tri$-comonoid structure.
In fact, in \cite{ahman2016directed}, Ahman and Uustalu showed that $\tri$-comonoids are just categories.
So $(\otimes, \tri)$-bimonoids can be thought of as \emph{collective categories}: categories for which there is a collective whose contributions are objects and whose returns are outgoing morphisms, and the structures cohere in a certain way.
Some of the examples of collectives in this paper are in fact collective categories, but the concept deserves further study.

On the examples side, we believe we have just scratched the surface, in terms of the sorts of examples that may exist and be of societal interest. New technologies like blockchain and DAOs offer the possibility of instantiating some of these ideas as working social structures. Other collectives could contribute automated problem-solving. For example, one could imagine a collective in which contributions were something like universal Turing machines. In algorithmic probability theory (also known as Solomonoff induction), one imagines predicting bit strings according to which ones can be implemented using the shortest programs; it turns out not to matter what language the programs are written in: they all are within a constant bound of each other in terms of prediction discrepancies. We do not need the details here, because the point is simply to imagine a collective of universal predictors. Each member contributes a prediction about their local environment, and these are aggregated to a prediction about the larger or more abstract environment. Errors in the aggregate prediction are distributed to the individual members. Clearly, the story is completely unfinished, and would benefit from attention by experts in Solomonoff induction, Bayesian learning, and so on; however we imagine that something of this sort---a collective structure on universal predictors---may be possible.

\printbibliography

\end{document}